\documentclass{amsart}
\title[On integrable generalizations of the pentagram map]{On integrable generalizations of the pentagram map}
\author{Gloria Mar\'i Beffa}
\usepackage{amsmath,amsfonts,amsthm}
\usepackage{graphicx,pstricks}
\usepackage{color}
\newtheorem{theorem}{Theorem}
\numberwithin{theorem}{section}

\newtheorem{lemma}[theorem]{Lemma}

\newtheorem{proposition}[theorem]{Proposition}

\theoremstyle{definition}
\newtheorem{definition}[theorem]{Definition}

\def\RP{\mathbb {RP}}
\def\R{\mathbb R}
\def\Z{\mathbb Z}

\def\PSL{\mathrm{PSL}}
\def\SL{\mathrm{SL}}
\def\GL{\mathrm{GL}}

\def\Sh{\sigma}

\def\p{{\bf p}}

\newcommand{\ai}[2]{a_{#1}^{#2}}
\def\wG{\widehat{G}}

\def\r{{\bf r}}

\def\ab{{\bf a}}
\def\w{{\bf w}}
\def\ve{{\bf v}}

\begin{document}
\maketitle
\begin{abstract} In this paper we prove that the generalization to $\RP^n$ of the pentagram map defined in \cite{KS} is invariant under certain scalings for any $n$. This property allows the definition of a Lax representation for the map, to be used to establish its integrability.\end{abstract}
\section{Introduction}
The pentagram map is defined on planar, convex $N$-gons. The map $T$ takes a vertex $x_k$ to the intersection of two segments: one is created by joining the vertices to the right and to the left of the original one, $\overline{x_{k-1}x_{k+1}}$, the second one by joining the original vertex to the second vertex to its right $\overline{x_kx_{k+2}}$ (see Fig. 1). These newly found vertices form a new $N$-gon. The pentagram map takes the first $N$-gon to this newly formed one. As surprisingly simple as this map is, it has an astonishingly large number of properties.

\vskip 2ex
\centerline{\includegraphics[height=1.3in]{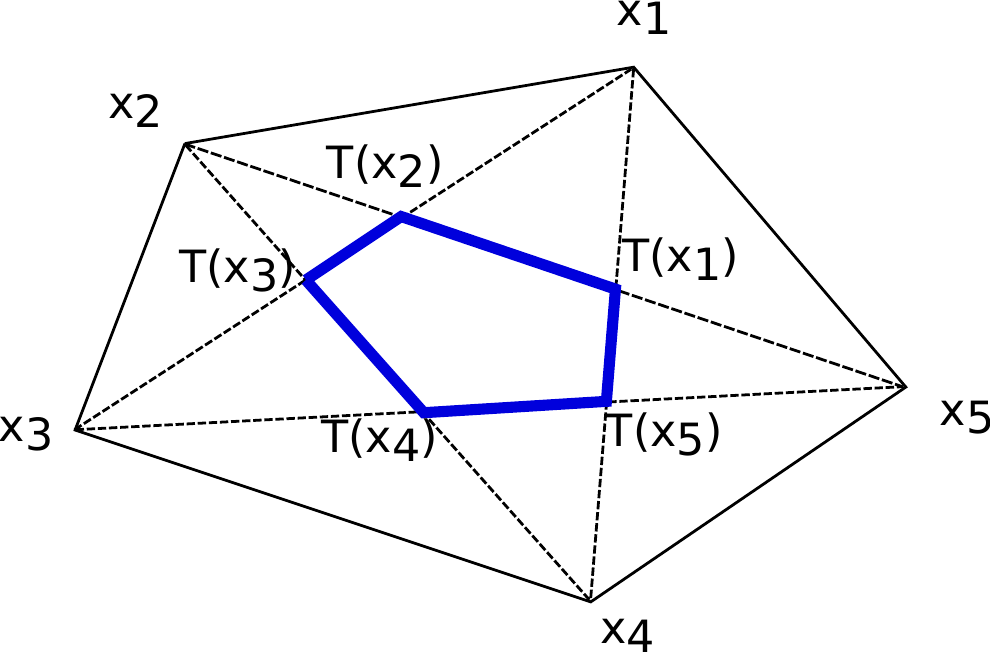}}
\vskip 1ex
\centerline{Fig. 1}
\vskip 2ex

The name pentagram map comes from the star formed in Fig. 1 when applied to pentagons. It is a classical fact that if $P$ is a pentagon, then $T(P)$ is projectively equivalent to $P$. It also seems to be classical that if $P$ is a hexagon, then $T^2(P)$ is projectively equivalent to $P$ as well. The constructions performed to define the pentagram map can be equally carried out in the projective plane. In that case { $T$ defined on the moduli space of pentagons (as described by the projective invariants of the polygons) is the identity, while defined on the moduli space of hexagons is an involution}. In general, one should not expect to obtain a closed orbit for any $N$; in fact orbits exhibit a quasi-periodic behavior classically associated to completely integrable systems. This was conjectured in \cite{S3}.

A recent number of papers (\cite{OST, OTS2, S1, S2, S3, ST, FS}) have studied the pentagram map and stablished its completely integrable nature, in the Arnold-Liouville sense. The authors of \cite{OST} defined the pentagram map on what they called  {\it $N$-twisted polygons}, that is, infinite polygons with vertices $x_k$, for which $x_{N+k} = M(x_k)$ for all $k$, where $M$ is the {\it monodromy}, a projective automorphism of $\RP^2$. They proved that, when written in terms of the {\it projective invariants} of twisted polygons, the pentagram map was in fact Hamiltonian and completely integrable. They displayed a set of preserved quantities and proved that almost every universally convex $N$-gon lie on a smooth torus with a $T$-invariant affine structure, implying that almost all the orbits follow a quasi-periodic motion under the map. The authors also showed that the pentagram map, when expressed in terms of projective invariants, is a discretization of the Boussinesq equation, a well-known  completely integrable system of PDEs. Integrability in the case of closed gons was proved in \cite{FS} and \cite{OTS2}.

The Boussinesq equation is one of the best known completely integrable PDEs. It is one of the simplest among the so-called Ader--Gel'fand-Dickey (AGD) flows. These flows are biHamiltonian and completely integrable. Their first Hamiltonian structure was originally defined by Adler in \cite{A} and proved to be Poisson by Gel'fand and Dickey in \cite{GD}, although the original definition is due to Lax. In \cite{M1} the author proved that some AGD flows of higher order can be obtained as the continuous limit of maps defined through the intersection of subspaces of different dimensions, but no complete integrability was proved. In \cite{KS} the authors studied the direct analogue of the pentagram map in $\RP^n$. It is defined by the intersection of $n$ hyperplanes obtained by shifting $n$ times a particular hyperplane containing every other vertex and as many vertices as needed. The authors proved that the map defined on the projective invariants of the polygons is completely integrable for $n=3$, describing also in detail the behavior of its orbits. They described a parameter-free Lax representation for the induced map on the invariants, and they conjectured that the map was invariant under scaling in the general case. This would usually guarantee the existence of  a standard Lax representation (and in that sense the integrability of the map) by introducing the scaling as spectral parameter in the parameter free one.  They conjectured the form of the scaling in all dimensions, proved it for $n=3$ and provided a computer-aided proof in dimensions $4$, $5$ and $6$ (the original version showed the wrong scaling for even dimensions, but this has been corrected to match the one in this paper).

In this paper we describe how any map induced on invariants by a map on polygons possesses a parameter-free Lax representation (in particular the pentagram map and all its generalizations). We then prove the conjecture in \cite{KS}, that is, there are scalings that leave the pentagram invariant in all dimensions, and they can be used to create the non-trivial Lax representation (by non-trivial we mean that the parameter cannot be removed by merely gauging it out). 

The proof of the invariance under the scaling is in fact a consequence of linear algebra: in section 2 we show how any map induced on invariants by a map on polygons possesses a parameter-free Lax representation, and using this representation we describe the pentagram map as the unique solution of a linear system of equations. Section 3 studies the homogeneity and degrees of the determinants involved in Cramer's rule, and uses the results to assert the invariance of the map. A key point that allows us to simplify the rather involved calculations is the fact that, even though the pentagram map in $\RP^n$ is described by intersecting $n$ hyperplanes, it can also be equally described as the intersection of only two planes of dimension $s$, if $n = 2s$, or two planes of dimensions $s$ and $s+1$ if $n = 2s+1$. This fact, explained in section 2, allows us to simplify the map so it can be studied in detail in section 3. Finally we conclude with a short discussion on the integrability of other possible generalizations.

This paper is supported by NSF grant DMS \#0804541 and a Simons Foundation's fellowship for 2012-13.

\section{Background and initial results}
\subsection{Discrete projective group-based moving frames}

 In this section we will describe basic definitions and facts needed along this paper on the subject of discrete group-based moving frames. They are taken from \cite{MMW} and occasionally slightly  modified to fit our needs. 

Let $M$ be a manifold and let $G\times M\to M$ be the action of a group $G$ on $M$. We will assume that $G \subset \GL(n,\R)$.
  \begin{definition}[Twisted $N$-gon]
{\it A twisted $N$-gon} in $M$ is a map $\phi:\Z\to M$ such that for some fixed $g\in G$ we have $ \phi(p+N) = g\cdot \phi(p)$ for all $p\in \Z$. (The notation $\cdot$ represents the action of $G$ on $M$.) The element $g\in G$ is called {\it the monodromy} of the polygon. We will denote a twisted $N$-gon by its image $x = (x_k)$ where $x_k = \phi(k)$.
 \end{definition}
The main reason to work with twisted polygons is our desire to have periodic invariants.
We will denote by $P_N$ the space of twisted $N$-gons in $M$. If $G$ acts on $M$, it also has a natural induced action on $P_N$ given by the diagonal action $g\cdot (x_k) = (g\cdot x_k)$.

 \begin{definition}[Discrete moving frame] Let $G^{N}$ denote the Cartesian product of $N$ copies of the group $G$. Elements of $G^N$ will be denoted by $(g_k)$. Allow $G$ to act on the left on $G^{N}$ using the  diagonal action
\(
g\cdot (g_k) = (g g_k).
\)
We say a map 
\[
\rho:P_N \to G^{N}
\] 
is a  left  {\it discrete moving frame} if $\rho$ is equivariant with respect to the action of $G$ on $P_N$ and the left diagonal action of $G$ on $G^{N}$. Since $\rho(x)\in G^{N}$, we will denote by $\rho_k$ its $k$th component; that is $\rho = (\rho_k)$, where $\rho_k(x) \in G$ for all $k$, $x = (x_k)$. \end{definition}

In short, $\rho$ assigns an element of the group to each vertex of the polygon in an equivariant fashion. For more information on discrete moving frames see \cite{MMW}.
 These group elements carry the invariant information of the polygon.
 
 \begin{definition}[Discrete invariant] Let $I:P_N \to \R$ be a function defined on $N$-gons. We say that $I$ is a scalar {\it discrete invariant} if
 \begin{equation}\label{invdef}
 I((g\cdot x_k)) = I((x_k))
 \end{equation}
 for any $g\in G$ and any $x = (x_k)\in P_N$.
 \end{definition}
 We will naturally refer to vector discrete invariants when considering vectors whose components are discrete scalar invariants.

 \begin{definition}[Maurer--Cartan matrix] 
Let ${\rho}$ be a  left discrete moving frame evaluated along a twisted $N$-gon. The element of the group
\[
K_k = \rho^{-1}_k\rho_{k+1} 
\]
is called the left  {\it $k$-Maurer--Cartan matrix} for $\rho$. We will call the equation $\rho_{k+1} = \rho_k K_k$ the  left  {\it $k$-Serret--Frenet equation}. 
\end{definition}
 The entries of a Maurer--Cartan matrix are functional generators of {\it all discrete invariants} of polygons, as it was shown in \cite{MMW}. 

 Assume we have a map $T: P_N \to P_N$, equivariant respect to the diagonal action, and let $(\rho_k)$ be a moving frame. Then
 \begin{equation}\label{invmap}
T(x)_k =\rho_k(x)\cdot \w_k(x)
\end{equation}
where $\w_k:P_N\to G/H$ is invariant, that is, $\w_k(g\cdot x) = \w_k(x)$ for any $g\in G$ (\cite{MMW}).

Finally, let us extend a map $T:P_N \to P_N$ to functions of $x=(x_k)$ the standard way using the pullback ($T(f(x)) = f(T(x)$). Let us also extend it to elements of the group by applying it to each entry of the matrix. We will denote the map with the same letter, abusing notation. Define the matrix
\begin{equation}\label{N}
N_k = \rho_k^{-1} T(\rho_k).
\end{equation}The following relationship with the Maurer-Cartan matrix is straightforward
\begin{equation}\label{structure}
T(K_k) =  T(\rho_k^{-1})T(\rho_{k+1}) = T(\rho_k)^{-1}\rho_k\rho_k^{-1}\rho_{k+1}\rho_{k+1}^{-1} T(\rho_{k+1}) = N_k^{-1} K_k N_{k+1}.
\end{equation}

Since $N_k$ is invariant, one can write it in terms of the invariants of the associated polygon. In fact, in most cases there are algorithms that achieve this by simply solving a system of equations, as shown in \cite{MMW}. That means that, if $K_k$ is, for example, affine on the invariants (which is the case for $\RP^n$, as we will see later),  (\ref{structure}) is a {\it parameter free Lax representation} of the map as defined on the invariants. The Serret--Frenet equations $\rho_{k+1} = \rho_k K_k$ together with $T(\rho_k) = \rho_k N_k$ define a parameter-free discrete AKNS representation of the map $T$. with the moving frame as its solution. This representation exists {\it for any} map induced on invariants by a map defined on polygons (a continuous version of this fact also exists, see \cite{M2}). 

Of course, integrability is usually achieved through a non-trivial Lax representation containing a spectral parameter. But if the map turns out to be invariant under a certain scaling, or indeed under the action of a $1$-parameter group, introducing that scaling in (\ref{structure}) as spectral parameter will give raise to a regular Lax representation - and in that sense to integrability. This is a well known approach to create Lax representations in integrable systems and it was used in \cite{KS} to generate the Lax representation for the cases they studied, conjecturing that it also existed for higher dimensional maps. Thus, once the invariance under scaling is proved for the higher dimensional cases, the existence of a Lax representation follows.

In the particular case of $G = \PSL(n+1, \R)$ and $M = \RP^n$ both moving frame and Maurer-Cartan matrices are well-known (see \cite{OST}). Given $(x_k) \in (\RP^n)^N$, and assuming non-degeneracy, we can find unique lifts of $x_k$ to $\R^{n+1}$, which we will call $V_k$, such that $\det(V_k, V_{k+1}, \dots, V_{k+n}) = 1$, for all $k$. One can do that {\it whenever $N$ and $n+1$ are co-primes}; a proof of this fact can be found in \cite{MW}, although it is probably a classical result. (To be precise, the lift is only unique for $n$ even. Since the monodromy is an element of $\PSL(n+1)$, to have it act on $\R^{n+1}$ one will have to fix a choice of monodromy in $\SL(n+1)$ - there are two such choices - before concluding that the lift is unique.)  Since the projective action becomes linear on lifts, the left projective moving frame is then given by 
\begin{equation}\label{projectivemf}
\rho_k = (V_k, V_{k+1}, \dots, V_{k+n}),
\end{equation}
which exists on non-degenerate polygons (by non-degenerate we mean that the lifts exist and are independent). Finally, since $\{V_k, V_{k+1}, \dots, V_{k+n}\}$ generate $\R^{n+1}$, there exist algebraic functions of $x$, we call them  $a_k^i$, such that 
\begin{equation}\label{Vinv}
V_{k+n+1} = a_k^n V_{k+n}+ a_k^{n-1} V_{k+n-1}+ \dots+ a_k^1 V_{k+1} + (-1)^n V_k
\end{equation}
where the coefficient $(-1)^n$ is imposed by the condition $\det \rho_k = 1$ for all $k$. Therefore, the Maurer-Cartan matrix is given by
\begin{equation}\label{projectivemcm}
K_k = \rho_k^{-1} \rho_{k+1} = \begin{pmatrix} 0&0&\dots & 0 & (-1)^n\\ 1&0&\dots&0&a_k^1\\ 0&1&\dots & 0& a_k^2\\ \vdots&\ddots&\ddots&\vdots&\vdots\\ 0&\dots&0&1&a_k^n\end{pmatrix}.
\end{equation}
Those familiar with the subject will recognize (\ref{projectivemcm}) as the discrete analogue of the Wilczynski invariants for projective curves (see \cite{Wi}).  
Finally,  using (\ref{invmap}) we conclude that any invariant map on polygons is defined on the lifts as
\begin{equation}\label{Veq}
T(V_k) = \rho_k \w_k
\end{equation}
for some invariant vector $\w_k\in \R^{n+1}$. To guarantee that $T$ is the lift of a projective map, these vectors have the extra condition
\begin{equation}\label{norm}
\det(\rho_k\w_k, \rho_{k+1}\w_{k+1}, \dots, \rho_{k+n}\w_{k+n}) = 1.
\end{equation}
Factoring $\rho_k$ and using the definition of the Maurer-Cartan matrix, the condition can be written as
\begin{equation}\label{Nmain}
\det(G_k, G_{k+1}, \dots, G_{k+n}) = 1
\end{equation}
where $G_k = \w_k$ and $G_{k+s} = K_k K_{k+1}\dots K_{k+s-1}\w_{k+s}$ for $s=1, \dots, n$. With this notation $N_k = \rho_k^{-1}T(\rho_k)$ can be written as
\begin{equation} \label{N-1}
N_k = \rho_k^{-1} T(\rho_{k}) = (G_k, G_{k+1}, \dots, G_{k+n}),
\end{equation}
a formulation central to our approach.

\subsection{The pentagram map and its generalizations}
 The pentagram map is defined on convex polygons on the projective plane, and it maps a vertex $x_k$ to the intersection of the segments $\overline{x_{k-1}x_{k+1}}$ and $\overline{x_k x_{k+2}}$. The new polygon is formed by these vertices. In the last couple of years some generalizations appeared in the literature: unsing the fact that the pentagram map is a discretization of the Bousinessq  equation (or $(2,3)$ AGD flow), the author studied in \cite{M2} possible discretizations of higher order AGD flows achieved by intersecting subspaces in $\RP^n$ of different dimensions. In \cite{KS} the authors defined a generalization of the pentagram map, defined as the intersection of $n$ hyperplanes in $\RP^n$ which are consecutive shifts of a particular hyperplane. As with the pentagram map, the particular hyperplane is obtained by joining every other vertex until it is uniquely determined. The authors of \cite{KS} showed that the continuous limit of this map is the higher dimensional Boussinesq equation (or $(2,n+1)$ AGD flow). We focus on the generalization in \cite{KS} which we describe next.
 
 Let $\{x_k\}$ be a {\it twisted} polygon in $\RP^n$. Define $P_k$ to be the hyperplane in $\RP^n$ containing the vertices:
 \begin{equation} \label{planeeven}
  x_{k-2s+1}, x_{k-2s+3}, \dots, x_{k-3}, x_{k-1}, x_{k+1}, x_{k+3}, \dots,  x_{k+2s-1}
 \end{equation}
 if $n = 2s$ is even;
\begin{equation} \label{planeodd} 
x_{k-2s}, x_{k-2s+2}, \dots, x_{k-2}, x_k, x_{k+2}, \dots, x_{k+2s-2}, x_{k+2s}
\end{equation} 
if $n = 2s+1$ is odd.
 
We will  assume that the polygon has the analogous property to being convex as in \cite{OST}:   the vertices above ($n$ of them in each case) uniquely determine the plane for all $k$ and the intersection of $n$ consecutive $P_k$'s determine a unique point in $\RP^n$. If the polygon has this property, and following \cite{KS}, we define the {\it generalized pentagram map} as
 \begin{equation}\label{pentaeven}
 T(x_k) = P_{k-s+1} \cap P_{k-s+2}\cap P_{k-s+3}\cap\dots\cap P_{k-1}\cap P_k \cap P_{k+1}\cap \dots\cap P_{k+s}
 \end{equation}
 if $n = 2s$ is even;
 \begin{equation}\label{pentaodd}
 T(x_k) = P_{k-s} \cap P_{k-s+1}\cap P_{k-s+3}\cap\dots\cap P_{k-1}\cap P_k \cap P_{k+1}\cap \dots\cap P_{k+s}
 \end{equation}
 if $n = 2s+1$ is odd.
 
Our approach to finding a scaling of $a_k^i$ preserved by $T$ relies on being able to write this map on the discrete projective invariants $a_k^i$ given in (\ref{projectivemcm}), expressing  it as solution of a linear system of equations. We proceed to do that next. 
Let $V_k$ be the lifted vertex as in (\ref{projectivemf}). We can describe $T(V_k)$ as follows:

\subsubsection{The $n = 2s$  case}

Since $T(x_k)$ is the intersection of hyperplanes, $T(V_k)$ will be the intersection of the lifted hyperplanes in $\R^{n+1}$. If $P_k$ is given as in (\ref{planeeven}), then its lifted one, we will call it $\Pi_k$, is spanned by the lifts of the corresponding vertices. That is, $\Pi_k$ is the subspace of $\R^n$ generated by the vectors
\begin{equation} \label{liftplaneeven}
  V_{k-2s+1}, V_{k-2s+3}, \dots, V_{k-3}, V_{k-1}, V_{k+1}, V_{k+3}, \dots,  V_{k+2s-1}.
 \end{equation}
 
 \begin{proposition} The lifted pentagram map $T(V_k)$ can be defined as the intersection of two $(s+1)$-dimensional subspaces in $\R^{2s+1}$, spanned by the vectors \[\{V_{k-s}, V_{k-s+2}, V_{k-s+4}, \dots, V_{k+s}\}\] and the vectors \[\{V_{k-s+1}, V_{k-s+3}, V_{k-s+5}, \dots, V_{k+s+1}\},\] respectively.
 \end{proposition}
 \begin{proof}
Let us select every other hyperplane and divide the complete set of $2s$ subspaces \[\{\Pi_{k-s+1}, \Pi_{k-s+2}, \dots, \Pi_{k+s}\}\] into two subsets  of $s$ hyperplanes each \[
\{\Pi_{k-s+1}, \Pi_{k-s+3}, \dots,  \Pi_{k+s-1}\} ~~ \mathrm{and}~~ \{\Pi_{k-s+2}, \Pi_{k-s+4}, \dots, \Pi_{k+s}\}.\] It is not too hard to see that the hyperplanes in each one of these two subsets intersect in an $s$-dimensional subspace. Indeed, $\Pi_{k-s+1}$ is generated by the vectors
\[
V_{k-3s+2}, V_{k-3s+4}, \dots, V_{k+s}
\]
while $\Pi_{k-s+3}$ will have these shifted twice to the right. We keep on shifting until we get to the last subspace in this group which has generators
\[
V_{k-s}, V_{k-s+2}, \dots, V_{k+3s-2}.
\]
Therefore, the intersection of the subspaces in this set contains $\{V_{k-s}, V_{k-s+2}, \dots, V_{k-s}\}$, with a total of $s+1$ vectors. Dimension counting tells us that the intersection is indeed spanned by them. Similarly one concludes that the set $  \{\Pi_{k-s+2}, \Pi_{k-s+4}, \dots, \Pi_{k+s}\}$ intersects at $\{V_{k-s+1}, V_{k-s+3}, V_{k-s+5}, \dots, V_{k+s+1}\}$, also of dimension $s+1$. Since the intersection of these two subspaces is $1$ dimensional, the intersection of the $2s$ hyperplanes equals the intersection of these two subspaces, as stated.
\end{proof}

Using this proposition we can describe $T(V_k)$ explicitly.
\begin{proposition}
If $n=2s$, the lifted pentagram map is given by
\[
T(V_k) = \lambda_{k-s} \rho_{k-s} \r_{k-s}
\]
where 
\begin{equation}\label{reven}
\r_k^T = \begin{pmatrix}1& 0& a_k^2& 0& a_k^4& 0& \dots& 0& a_k^{2s}\end{pmatrix},
\end{equation}
and where $\lambda_k$ is uniquely determined by condition (\ref{norm}) with $\w_k = \lambda_k\r_k$.\end{proposition}
\begin{proof}
From our previous proposition we know that $T(V_k)$ can be written as linear combinations of both $\{V_{k-s}, V_{k-s+2}, \dots, V_{k+s}\}$ and $\{V_{k-s+1}, V_{k-s+3}, V_{k-s+5}, \dots, V_{k+s+1}\}$. But since $V_{k-s}, V_{k-s+1}, \dots, V_{k+s}$ form a basis for $\R^{2s+1}$, and given (\ref{Vinv}) we have that there exist invariants $\alpha_i^j$ and $\beta_i^j$ such that
\begin{eqnarray*}
T(V_k) &=& \alpha_k^1 V_{k-s} + \alpha_k^2V_{k-s+2}+\dots +\alpha_k^{s+1} V_{k+s} \\ &=& \beta_k^1V_{k-s+1}+\beta_k^2 V_{k-s+3}+\dots+\beta_k^sV_{k+s-1} \\&+& \beta_k^{s+1}\left(a_{k-s}^{2s}V_{k+s}+a_{k-s}^{2s-1}V_{k+s-1}+\dots+a_{k-s}^1V_{k-s+1} + V_{k-s}\right).
\end{eqnarray*}
 Equating the coefficients in this basis and choosing $\lambda_{k-s} = \beta_k^{s+1}$ we obtain the first part of the proposition.  The second part is immediate since $T$ is lifted from a map in projective space, hence the lift is well defined whenever it satisfies (\ref{norm}) for all $k$. That is
 \begin{equation}\label{lambdaeq}
 \lambda_k\lambda_{k+1} \dots \lambda_{k+2s}\det(\rho_k\r_k, \rho_{k+1}\r_{k+1}, \dots, \rho_{k+2s}\r_{k+2s}) = 1.
 \end{equation}
 As before, this equation has a unique solution for $\lambda_k$ whenever $N$ and $n+1 = 2s+1$ are co-primes.
\end{proof}
The elements $\lambda_k$ are the same as those found in \cite{KS} through a different formulation. \subsubsection{The case $n=2s+1$}
In the odd dimensional case the results are similar, with different choices of vertices and hyperplanes. Since the proofs of the propositions are identical we will omit them. 
 If $P_k$ is given as in (\ref{planeodd}), then $\Pi_k$ is the subspace of $\R^n$ generated by the vectors
\begin{equation} \label{liftplaneodd}
  V_{k-2s}, V_{k-2s+2}, \dots, V_{k-2}, V_{k}, V_{k+2}, \dots, V_{k+2s-2}, V_{k+2s}.
 \end{equation}
 First we describe the two subspaces intersecting in $T(V_k)$.
 \begin{proposition} The lifted pentagram map $T(V_k)$ can be defined as the intersection of one $s+1$-dimensional and one $s+2$-dimensional subspaces in $\R^{2s+2}$, spanned by the vectors \[\{V_{k-s}, V_{k-s+2}, V_{k-s+4}, \dots, V_{k+s}\}\] and the vectors \[\{V_{k-s-1}, V_{k-s+1}, V_{k-s+3}, \dots, V_{k+s+1}\},\] respectively.
 \end{proposition}

Next we write the lift of the pentagram map in terms of the moving frame.
\begin{proposition}
If $n=2s+1$, the lifted pentagram map is given by
\[
T(V_k) = \lambda_{k-s} \rho_{k-s} \r_{k-s}
\]
where 
\begin{equation}\label{rodd}
\r_k^T = \begin{pmatrix} 0&  a_k^1& 0& a_k^3& 0& \dots& 0& a_k^{2s+1}\end{pmatrix},
\end{equation}
and where $\lambda_k$ is uniquely determined by condition (\ref{norm}) with $\w_k = \lambda_k\r_k$.\end{proposition}

\vskip 2ex

From now on, and for simplicity's sake, we will shift the pentagram map and redefine it as $\Sh^s T$ where $\Sh$ is the shift map, $\Sh(x_k) = x_{k+1}$  extended to functions of $x$ via the pullback and to vectors and matrices applying it to every entry. With this $T(V_k) = \lambda_k\rho_k\r_k$. Since the scaling will be independent of $k$, this will produce no trouble in what follows.

To end this section we will show that, in both odd and even dimensional cases, we can rewrite the map on the invariants as a solution of a linear systems of equations. Indeed, from (\ref{N-1})
\begin{equation}\label{N0}
N_k= \begin{pmatrix}\lambda_k\r_k& \lambda_{k+1}K_k \r_{k+1}& \lambda_{k+2}K_kK_{k+1}\r_{k+2}& \dots& \lambda_{k+n} K_kK_{k+1}\dots K_{k+n-1}\r_{k+n}\end{pmatrix}
\end{equation}
Also, using (\ref{structure}), we have that
\[
N_k T(K_k) = K_kN_{k+1} = K_k\left(\lambda_{k+1}\r_{k+1}, \lambda_{k+2}K_{k+1}\r_{k+2}, \dots, \lambda_{k+n+1}K_{k+1}\dots K_{k+n}\r_{k+n+1}\right).
\]
Let us call $F_k = \r_k$ and 
\begin{equation}\label{F}
F_{k+s} = K_k\dots K_{k+s-1}\r_{k+s},
\end{equation}
for $s=1,2\dots$. Then the above can be written as
 \[
 N_k T(K_k) = (\lambda_{k+1}F_{k+1},\lambda_{k+2}F_{k+2}, \dots, \lambda_{k+n+1} F_{k+n+1}).
 \]
Since the last column of $N_kT(K_k)$ is given by $N_k\begin{pmatrix} 1\\ T(\ab_k)\end{pmatrix}$, with $\ab_k$ indicating the vectors of invariants, we have that $T(\ab_k)$ is the unique solution to the equation

\begin{equation}\label{main}
N_k \begin{pmatrix} 1\\ T(\ab_k)\end{pmatrix} = \lambda_{k+n+1}F_{k+n+1}.
\end{equation}

Studying the scaling invariance of a solution of a linear system of equations is equivalent to studying the invariance of the determinants appearing in Cramer's rule. This is what we will do in our next section.
\section{Scaling invariance of the generalized pentagram map}\label{AGD}

In this section we also need to separate the odd and even cases. The scaling and proof in the odd case is simpler, so we will present it first. 

If a function is homogeneous of degree $r$ under a scaling we say $d(f) = r$. Likewise for vectors and matrices whose entries all have the same degree. If entries of a vector $\ve$ are homogeneous with different degrees, we will write $d(\ve)$ as a vector of numbers, each one the degree of the corresponding entry.
\subsection{The case of $\RP^{2s+1}$}

Assume $n = 2s+1$ and consider the scaling
\begin{equation}\label{scalodd}
a_k^{2\ell+1} \to t a_k^{2\ell+1} \hskip 2ex  \hskip 4ex a_k^{2\ell} \to a_k^{2\ell}, \hskip 2ex \ell =1, \dots s.
\end{equation}
This scaling appeared in \cite{KS} where the authors conjectured that it left the pentagram map invariant, a fact they proved for $s=1$ and checked for $s=2$ with the aid of a computer. 

Before we start we will prove a simple but fundamental lemma.
\begin{lemma} Let $F_{k+i}$ be defined as in (\ref{F}). Then 
\begin{equation}\label{induction}
F_{k+2\ell} = \sum_{r=1}^\ell \alpha_{2r-1}^{2\ell} F_{k+2r-1} + G_{k+2\ell}, \hskip 2ex F_{k+2\ell+1} = \sum_{r=0}^\ell\alpha_{2r}^{2\ell+1} F_{k+2r} + \wG_{k+2\ell+1}
\end{equation}
for $\ell = 1,2,\dots$, where
\begin{equation}\label{alpha1}
\wG_{k+2\ell+1} = \left(\Sh G_{k+2\ell}\right)_{n+1}\p_k+ \left[\Sh G_{k+2\ell}\right]^1,\hskip 2ex
\alpha_{2r}^{2\ell+1} = \Sh\alpha_{2r-1}^{2\ell}, \hskip 1ex \alpha_0^{2\ell+1} = \left(\Sh G_{k+2\ell}\right)_{n+1}.
\end{equation}
and 
\begin{equation}\label{alpha2} G_{k+2\ell+2} = \left[\Sh \wG_{k+2\ell+1}\right]^1, \hskip 2ex \alpha_{2r+1}^{2\ell+2} = \Sh\alpha_{2r}^{2\ell+1}.
\end{equation}
We denote by $(~)_{n+1}$ the last entry of a vector and $[~]^1$ indicates that the vectors' entries have been shifted downwards once and a zero has been added in the first entry. 
\end{lemma}
By a hat we denote vectors whose last entry vanishes, while an absence of hat means the last entry is not zero. In both cases the entries vanish alternatively, so non-hat vectors have zero odd entries and hat vectors have zero even entries .  For simplicity we will drop the $k$ subindex in proofs and denote by $K_r$ the matrix $K_{k+r}$. Likewise with the other elements in the calculation. We will only introduce the subindex $k$ if its omission might create confusion.
\begin{proof} Recall that $\r$ is given by (\ref{rodd}). Denote by $\p$ the vector $\p = (-1,0,a^2,0\dots,  a^{2s},0)^T$ so that the last column of $K$ in (\ref{projectivemcm})  is given by $\p+\r$. Notice also that from the definition of $F$ in (\ref{F}), 
\[
F_\ell = K \Sh F_{\ell-1}.
\]
Using this, we have
\[
F_1 = K \r_1 = a_1^{2s+1}(\p+\r) + [\r_1]^1 = a_1^{2s+1}F+a_1^{2s+1}\p+[\r_1]^1
\]
We call $\wG_1 = a_1^{2s+1}\p+[\r_1]^1$. Next
\[
F_2 = K\Sh F_1 = K(a_2^{2s+1}\Sh F + \Sh \wG_1) = a_2^{2s+1} F_1 + \left[\Sh \wG_1\right]^1
\]
and we call $G_2 =  \left[\Sh \wG_1\right]^1$.  Assume that
\[
F_{2\ell} = \sum_{r=1}^\ell \alpha_{2r-1}^{2\ell} F_{2r-1} + G_{2\ell}.
\]
Then, 
\[
F_{2\ell+1} = K\Sh F_{2\ell} = \sum_{r=1}^\ell \Sh \alpha_{2r-1}^{2\ell} K\Sh F_{2r-1} + K \Sh G_{2\ell}.
\]
Since $K \Sh G_{2\ell} = \left(\Sh G_{2\ell}\right)_{n+1}\p + \left(\Sh G_{2\ell}\right)_{n+1}\r + \left[\Sh G_{2\ell}\right]^1$ and $\r = F$, we have
\[
F_{2\ell+1} =  \sum_{r=0}^\ell\alpha_{2r}^{2\ell+1} F_{2r} + \wG_{2\ell+1}
\]
with 
\[
\wG_{2\ell+1} = \left(\Sh G_{2\ell}\right)_{n+1}\p+ \left[\Sh G_{2\ell}\right]^1,\hskip 2ex
\alpha_{2r}^{2\ell+1} = \Sh\alpha_{2r-1}^{2\ell}, r=1,\dots, \ell; \hskip 1ex \alpha_0^{2\ell+1} = \left(\Sh G_{2\ell}\right)_{n+1}.
\]
Going further one step we have
\[
F_{2\ell+2} = K\Sh F_{2\ell+1} = \sum_{r=0}^\ell \Sh\alpha_{2r}^{2\ell+1} K\Sh F_{2r} + K\Sh \wG_{2\ell+1}
\]
and since $K\Sh F_{2r} = F_{2r+1}$ and $K\Sh \wG_{2\ell+1} = \left[\Sh \wG_{2\ell+1}\right]^1$, calling 
\[ G_{2\ell+2} = \left[\Sh \wG_{2\ell+1}\right]^1 \hskip 2ex \mathrm{and}\hskip 2ex \alpha_{2r+1}^{2\ell+2} = \Sh\alpha_{2r}^{2\ell+1}
\]
concludes the proof of the lemma.
\end{proof}

\begin{theorem} The determinant  
\begin{equation}\label{N2}
D_k = \det\begin{pmatrix}F_k& F_{k+1}& F_{k+2}& \dots& F_{k+2s+1}\end{pmatrix}
\end{equation}
where $F_k$ is given as in (\ref{F}) and $\r_k$ is given as in (\ref{rodd}), is homogeneous under the scaling (\ref{scalodd}), with $d(D_k) = 2s+2$ for all $k$. Furthermore, if $\lambda_k$ is the solution of (\ref{lambdaeq}), then $\lambda_k$ is also homogeneous with $d(\lambda_k) = -1$ for all $k$.
\end{theorem}
\begin{proof}

We will first show that $d(G_{2\ell}) = d(\wG_{2\ell+1}) = 1$ for all $\ell$, where $G_{2\ell}$ and $\wG_{2\ell+1}$ are defined as in the previous lemma. (Notice that we are implying that zero has also degree $1$. In fact, zero has {\it any} degree and it is in that sense that we claim $d(G_{2\ell}) = d(\wG_{2\ell+1}) = 1$.) First of all $\wG_1 = a_i^{2s+1}\p+[r_1]^1$ with $d(\p) = 0$, $d(a_i^{2s+1}) = d(\r_1) = 1$. Furthermore $G_2 = \left[\Sh\wG_1\right]^1$. Therefore, $d(G_2) = d(\wG_1) = 1$. Since $d(G_{2\ell}) = d(\left[\Sh\wG_{2\ell-1}\right]^1)$ and $d(\wG_{2\ell+1}) = d((G_{2\ell})_{n+1}\p + \left[\Sh G_{2\ell}\right]^1) = d(G_{2\ell})$, we readily see that $d(G_{2\ell}) = d(\wG_{2\ell+1}) = 1$ for all $\ell$. 

Using this fact and the previous lemma, we see that
 since all the columns of the determinant $D$ have homogeneous degree $1$, $D$ will be homogeneous of degree $2s+2 = n+1$, which is the first assertion of the statement.

 The second assertion is also rather simple. Since according to (\ref{lambdaeq}) $\lambda$ is a solution of
\[
\lambda_0 \lambda_1\dots\lambda_n = D^{-1}
\]
we can apply logarithms (adjusting signs if necessary) and transform this equation into (we are re-introducing $k$ as here it is needed)
\[
\sum_{r = 0}^n\eta_{k+r} = -\ln D_k, \hskip 4ex \eta_{k+r} = \ln \lambda_{k+r},
\]
for any $k = 0, 1, \dots, N-1$. If $N$ and $n+1$ are coprimes, this system of  equations has a unique solution for any $D$ (see \cite{MW}). If we apply the scaling, the transformed $\eta_{k+r}$, which we denote $\tilde \eta_{k+r}$ will be the solution of the system
\[
\sum_{r = 0}^n\tilde\eta_{k+r} = -\ln D_k - \ln t^{n+1}, \hskip 2ex k=0,1,\dots,N-1.
\]
Therefore $\tilde\eta_{k+r} = \eta_{k+r} + \nu_{k+r}$, where $\nu_{k+r}$ satisfies the system
\[
\sum_{r = 0}^n\nu_{k+r} = - \ln t^{n+1}, \hskip 2ex k=0, 1, \dots, N-1. 
\]
Clearly $\nu_{k+r} = -\frac1{n+1} \ln t^{n+1}$ for all $r=0,\dots,n$ are solutions, and hence $\tilde\eta_{k+r} = \eta_{k+r} -\frac1{n+1} \ln t^{n+1} = \ln t^{-1}\lambda_{k+r}$. Thus, $d(\lambda_k) = -1$ for all $k$ and the theorem follows.
\end{proof}

We finally arrive at our main result.
\begin{theorem} The pentagram map on $\RP^{2s+1}$ is invariant under the scaling (\ref{scalodd}).
\end{theorem}
\begin{proof} As in the previous proof, we will drop the subindex $k$ unless needed. From (\ref{main}) and using Cramer's rule, we know that
\begin{equation}\label{Todd}
T(a^i) = \frac{\lambda_{2s+2}}{\lambda_i} \frac{D^i}D
\end{equation}
where $D^i$ is equal to
\[
D^i = \det(F, F_1, \dots, F_{i-1}, F_{2s+2}, F_{i+1}, \dots, F_{2s+1})
\]
$i=1, \dots, n$ and $D$ is a in (\ref{N2}). Recall that $F_{2s+2}$ has a recursion formula given in (\ref{induction}). We will separate the odd and even cases.

{\it Case $i=2\ell+1$ odd}. Using (\ref{induction}), noticing that $F_i$ appears only in the expansions  of $F_r$, $r$ even,  and simplifying we can write $D^i$ as
\begin{eqnarray*}
D^i &=& \alpha_i^{2s+2} D + \det(F, F_1, \dots, F_{i-1}, G_{2s+2}, F_{i+1}, \dots F_n) =  \alpha_i^{2s+2} D \\
 &+& \sum_{r=1}^{s-\ell} \alpha_i^{i+2r+1}\det(F, \dots, F_{i-1}, G_{2s+2}, G_{i+1}, \wG_{i+2}, \dots, \wG_{i+2r}, F_i, \wG_{i+2r+2}, \dots \wG_{2s+1})\\
&+& \det(F, F_1, \dots, F_{i-1}, G_{2s+2}, G_{i+1}, \dots, \wG_{2s+1}).
\end{eqnarray*}
The last term is equal to 
\[
\det(F, \wG_1, \dots, G_{i-1}, G_{2s+2}, G_{i+1}, \dots, \wG_{2s+1})
\]
which is zero since we have $s+2$ columns without hats spanning an $s+1$ dimensional space. The middle sum is equal to 
\[
\sum_{r=1}^{s-\ell} \alpha_i^{i+2r+1}\det(F,\wG_1, \dots, G_{i-1}, G_{2s+2}, G_{i+1}, \wG_{i+2}, \dots, \wG_{i+2r}, \wG_i, \wG_{i+2r+2}, \dots \wG_{2s+1})
\]
which is not zero since we can exchange $ \wG_i$ with $G_{2s+2}$ to have alternating hat and non-hat vectors. But we know that the determinant has columns of degree $1$, and hence it has degree $2s+2$.  We will prove that $\alpha_i^j$ are homogeneous and $d(\alpha_i^j) = 1$. This will imply that $D^i$ is homogeneous and $d(D^i) = 2s+3$. Taking this to (\ref{Todd}) and using the previous theorem we have that $T(a^i)$ is homogeneous and $d(T(a^i)) = 1 = d(a^i)$ as desired. We can see that $\alpha_i^j$ is homogeneous of degree $1$ directly from (\ref{alpha1}) and (\ref{alpha2}). Indeed, the beginning value $\alpha_0^{2\ell+1}$ is the last entry of $\Sh G_{2\ell}$, and hence $d(\alpha_0^{2\ell+1}) = 1$. Subsequent coefficients are given by shifts of these, and hence they are all homogenous of degree $1$.

{\it Case $i = 2\ell$ even}. 
If $i$ is even, then $F_i$ will not appear in the expansion (\ref{induction}) for $F_{2s+2}$. Therefore, substituting the expansions from right to left we obtain
\[
D^i =  \det(F, \dots, F_{i-1},G_{2s+2}, F_{i+1}, \dots, F_n)   
\]
\[
= \sum_{r=1}^{s-\ell+1} \det(F, \dots, F_{i-1}, G_{2s+2}, \wG_{i+1}, G_{i+2}, \dots, G_{i+2r-2}, \wG_{i+2r-1} + \alpha_i^{i+2r-1}F_i, G_{i+2r}, \dots \wG_{2s+1}).
\]
Now, splitting the terms and applying (\ref{induction}) recurrently starting with $F_i$, we can rewrite it as
\[
D^i = \det(F, \wG_1, \dots, \wG_{i-1}, G_{2s+2}, \wG_{i+1}, \dots, \wG_{2s+1})  +
\]
\[
\sum_{r=1}^{s-\ell+1} \alpha_i^{i+2r-1}\det(F, \wG_1, \dots, \wG_{i-1}, G_{2s+2}, \wG_{i+1}, G_{i+2}, \dots, G_{i+2r-2},  G_i, G_{i+2r}, \dots \wG_{2s+1}).
\]
But the last terms all vanish since again we have s+2 non-hat vectors generating an $s+1$ dimensional subspace. Hence
\[
D^i = \det(F, \wG_1, \dots, \wG_{i-1}, G_{2s+2}, \wG_{i+1}, \dots, \wG_{2s+1})  
\]
and $d(D^i) = d(D) = 2s+2$. Thus, from (\ref{Todd}) and the lemma we have $d(T(a^i)) = 0 = d(a^i)$.
\end{proof} 

\subsection{The case of $\RP^{2s}$} 

In the even case the scaling that leaves the pentagram map invariant is more involved than in the odd case, for reasons that will be clear along our calculations. It is given by 
\begin{equation}\label{scaleven}
 \ai k{2\ell+1} \to t^{-1+\ell/s} \ai k{2\ell+1}, \ell=0,\dots s-1,\hskip 4ex \ai k{2\ell}\to t^{\ell/s}\ai k{2\ell}, \ell=1,\dots s
\end{equation}
As in the odd dimensional case, to prove the scaling is preserved by the pentagram map, we will

\noindent a.  Prove that the determinant of the matrix
\begin{equation}\label{N3}
D_k =\det \begin{pmatrix} \r_k& F_{k+1} & F_{k+2} & \dots &F_{k+2s-1} & F_{k+2s}\end{pmatrix}
\end{equation}
where $F_i$ is given as in (\ref{F}), is homogeneous with $d(D_k) = 0$ for all $k$. (From the analogue to (\ref{lambdaeq}) in the even case, this implies that $\lambda_k$ are invariant under the scaling.); and

\noindent b.  Prove that the Cramer determinants $D^i$ associated to  $T(\ai k i)$ are also homogeneous with degree that match that of the corresponding $\ai k i$. 

The proof of these two facts are given in the main two theorems of this section, but first a simple lemma.

\begin{lemma} Assume an $r\times r$ matrix $B$ has entries $b_{i j}$ with homogeneity degrees equal to $d(b_{i j}) = \frac {i-j}s$. Then $\det B$ is invariant under (\ref{scaleven}).
\end{lemma}
\begin{proof} The determinant is the sum of $\pm$ products of $r$ entries, no two of them in the same row or column. Hence, the degree of each one of those products will be the sum of the degrees of the entries involved. That means we will be adding all $i's$ and all $j's$, resulting on a degree equal to
$\sum_{i=1}^r\frac is - \sum_{j=1}^r \frac js = 0$
\end{proof}

\begin{theorem} If $n = 2s$, the determinant of the matrix (\ref{N3}) is invariant under the scaling (\ref{scaleven}).
\end{theorem}
\begin{proof}  
 Recall that in the even dimensional case case $\r_k$ is given by
\begin{equation}\label{reven}
\r_k^T = \begin{pmatrix} 1 & 0 & a^2_k & 0& a_k^4 & \dots &0 & a_k^{2s}\end{pmatrix}^T.
\end{equation}
As before we will denote by $\p_k$ the vector containing the odd invariants so that  $\p_k+ \r_k$ is equal to the last column of $K_k$. That is, we are exchanging the roles of $\r_k$ and $\p_k$ in the previous section so that $\r_k$ contains now odd non-zero entries and even invariants, while $\p_k$ contains even nonzero entries and odd invariants. Also as before we will drop the subindex $k$ unless confusing.

First of all notice that the formulas (\ref{induction}), (\ref{alpha1}) and (\ref{alpha2}) can be obtained independently of the dimension and equally with our new choice of $\r$ and $\p$, since the only condition on $\r$ and $\p$ that was used was that $\p+\r$ gives the last column of $K$. Of course, in this case the degree of $G_i$ and $\wG_i$ will be different since we have a different scaling, but all other relations hold true with the new choices of $\p$ and $\r$, including the fact that hat-vectors have zero last entry and non hat vectors don't, and the fact that hat vectors have the same non-zero entries as $\p$, while non hat vectors have the same non-zero entries as $\r$.

Using those equations, we can  write the determinant of (\ref{N3}) as the determinant of the matrix
\[
D = \det\begin{pmatrix} \r & \wG_{1}& G_{2} & \wG_{3} & \dots & \wG_{2s-1} & G_{2s}\end{pmatrix}.
\]
Notice next that the first row contains only zeroes, with the exception of the first entry which is a $1$, since (unlike the previous case) $\p$ has zero first entry. Hence the determinant reduces to 
\[
\det \begin{pmatrix} \wG_{1}& G_{2} & \wG_{3} & \dots & \wG_{2s-1} & G_{2s}\end{pmatrix}.
\]
Here we are abusing the notation by denoting the vectors with the same letter, even though we are ignoring their first entry.  Let us denote by $g_{r}$ the $s$-vector formed by the nonzero entries of $\wG_{r}$. Since $G_{2\ell} = \left[\Sh \wG_{2\ell-1}\right]^1$, we can conclude that
\[
D= \Delta \Sh \Delta
\]
where $\Delta = \det(g_{1}, g_{3}, \dots,  g_{2s-1})$.

To finish the proof, we will show that the entry $(i,j)$ of $\Delta$ is homogeneous of degree $\frac{i-j}s$, and we will apply the lemma. Combining (\ref{alpha1}) and (\ref{alpha2}) we get that
\[
\wG_{2\ell+1} = \left(\Sh G_{2\ell}\right)_{n+1}\p + \left[\Sh G_{2\ell}\right]^1 = \left(\left[\Sh^2 \wG_{2\ell-1}\right]^1\right)_{n+1}\p + \left[\Sh^2  \wG_{2\ell-1}\right]^2.
\]
This relation can be translated to $g_{2\ell+1}$ as
\[
g_{2\ell+1} = \left(\Sh^2 g_{2\ell-1}\right)_{s}\bar{\p} + \left[\Sh^2 g_{2\ell-1}\right]^1
\]
with $\bar{~}$ indicating that we have removed the zero entries, so $\bar{\p}^T = (a^1, a^3, \dots, a^{2s-1})^T$ and likewise with $\bar{[\r]}^1$.
 We calculate the degree of $g_{1}$ first: since $\wG_{1} = a_{1}^{2s}\p + [\r_{1}]^1$, we have that $g_{1} = a_{1}^{2s}\bar{\p} + \bar{[\r_{1}]}^1$. According to (\ref{scaleven}), $d(a_{1}^{2s}) = 1$ and the entries of $\bar{\p}$ have degrees $d(\bar{\p}) =(-1, -1+1/s, -1+2/s, \dots, -1/s)^T$; therefore $d(a_{1}^{2s}\bar{\p}) = (0, 1/s, 2/s, \dots, (s-1)/s)$. Likewise $\bar{[{\r}_{1}]}^1$ has a constant in the first entry (of degree $0$) and degrees $(\ast, 1/s, 2/s, \dots, (s-1)/s)^T$ in the other entries. From here, $d(g_1) = (0, 1/s, 2/s, \dots, (s-1)/s)$, which coincide with $(i-j)/s$ for the entry in place $(i,j)$, since this is the first column.
 
 We now follow with a simple induction: assume that $d(g_{2\ell-1}) = (-(\ell-1)/s, -(\ell-2)/s, \dots, -1/s, 0, 1/s, 2/s, \dots, (s-\ell)/s)^T$. Then, since the degree of $(g_{2\ell-1})_{s}$ is $\frac{s-\ell}s$, $d(\bar{\p}) = (-1, -1+1/s, -1+2/s, \dots, -1/s)^T$, and $d(\left[\Sh^2g_{2\ell-1}\right]^1) = (\ast, \frac{-\ell+1}s, \frac{-\ell+2}s, \dots, \frac{s-\ell-1}s)^T$, we have 
 \begin{equation}\label{gordereven}
 d(g_{2\ell+1}) = d\left(\left[\Sh^2g_{2\ell-1}\right]^1 + \Sh^2 (g_{2\ell-1})_{s-1} \bar{\p}\right) = \begin{pmatrix}-\frac{\ell}s\\ \frac{-\ell+1}s\\\frac{-\ell+2}s \\\vdots\\ \frac{s-\ell-1}s\end{pmatrix},
 \end{equation}
which coincides with $\frac{i-j}s$ since $g_{2\ell+1}$ is the $\ell+1$ column. This ends the proof.
\end{proof}

The last theorem of this paper shows that the Cramer determinant corresponding to $T(a_{k}^i)$ in the linear equation (\ref{main}) has scaling equal to that shown in (\ref{scaleven}), implying that $T$ is invariant under (\ref{scaleven}).

\begin{theorem} Let
\[
D_{k}^{i} = \det\begin{pmatrix} \r_k& F_{k+1} & \dots & F_{k+i-1}& F_{k+2s+1} & F_{k+i+1} & \dots F_{k+2s} \end{pmatrix}
\]
for $i=1,2,\dots, 2s$. Then, $D^i_k$ is homogeneous and the degree of $D_{k}^{i}$ with respect to the scaling (\ref{scaleven}) is given by 
\begin{eqnarray*}
d(D_{k}^{2\ell-1}) &=& -1 + \frac{\ell-1}s = d(a_k^{2\ell-1}),\\ d(D_{k}^{2\ell}) &=& \frac \ell s = d(a_k^{2\ell}), 
\end{eqnarray*}
for $\ell = 1, \dots s$.
\end{theorem}
Using this, the previous theorem, and the fact that from (\ref{main})
\[
T(a_k^i) = \frac{\lambda_{k+2s+1}}{\lambda_{k+i}}\frac{D_k^i}{D_k}
\]
we conclude that (\ref{scaleven}) preserves $T$.
\begin{proof} 

{\it Case $i = 2\ell-1$ odd}. As in the proofs of previous theorems we will drop the subindex $k$. Using the relations (\ref{induction}), and the fact that $F_i$ does not appear in the expansion of $F_{2s+1}$, we can reduce $D^i$ to the expression
\[
D^i = \det\begin{pmatrix} \r& F_{1} & F_{2}& \dots & F_{i-1} & \wG_{2s+1}& F_{i+1} & F_{i+2} & \dots  &  F_{2s}\end{pmatrix}.
\]
Now, every $F_{i+r}$, with $r$ odd, can be substituted by $G_{i+r}+\alpha_i^{i+r} F_i$, while every $F_{i+r}$ with $r$ even can be substituted by $\wG_{i+r}$ since $F_i$ does not appear in its expansion. Substituting from right to left we get
\[
D^i = \det\begin{pmatrix}  \r & \dots & F_{i-1} & \wG_{2s+1}& a_{i}^{i+1} F_i + G_{i+1} & \wG_{i+2} &a_{i}^{i+3} F_i + G_{i+3} & \wG_{i+4} \dots  &  a_{i}^{2s} F_i+G_{2s}\end{pmatrix}
\]
\[
=\det\begin{pmatrix}  \r& \dots & F_{i-1} & \wG_{2s+1}& a_{i}^{i+1} \wG_i + G_{i+1} & \wG_{i+2} &a_{i}^{i+3} \wG_i + G_{i+3} & \wG_{i+4} \dots  &  a_{i}^{2s} \wG_i+G_{2s}\end{pmatrix}.
\]
Given that $\wG_{1}, \wG_{3}, \dots, \wG_{i-2}, \wG_{2s+1},\wG_{i+2},\dots, \wG_{2s-1}$ have, generically, full rank and $\wG_i$ belongs to this subspace, we can remove $\wG_i$ from the formula to obtain
\[
D^i = \det\begin{pmatrix}  \r& \wG_{1} & G_{2}& \dots & G_{i-1} & \wG_{2s+1}& G_{i+1} & \wG_{i+2} & \dots  &  G_{2s}\end{pmatrix}.
\]
Let $W^{i} = \wG_{2s+1}\otimes \wG_{i}^{-1} \otimes \wG_{i}$ represent the vector obtained by multiplying each one of the entries of $\wG_{2s+1}$ by the corresponding entry of $\wG_{i}$ and its inverse; we see that the value of $D^{i}$ does not change if we substitute $\wG_{2s+1}$ by $W^{i}$. But since (\ref{N3}) is invariant under scaling, we conclude that, if the entries of $ \wG_{k+2s+1}\otimes \wG_{k+i}^{-1} $ have all equal degree given by $d^i$, then $D^i$ is homogeneous and
\[
d(D^{i}) = d^i. 
\]
And indeed all entries have equal degree since, according to (\ref{gordereven}), $d(g_{2s+1}\otimes g_{i}^{-1}) = d(g_{2s+1})- d(g_{i}) = (-1, -1+1/s, \dots ,-1/s) - ( -(\ell-1)/s, -(\ell-2)/s, \dots, -1/s,0,1/s,\dots, (s-\ell)/s) = (-1+(\ell-1)/s, -1+(\ell-1)/s, \dots, -1+(\ell-1)/s)$. Therefore,  $d^i = -1+(\ell-1)/s$ and 
\[
d(D^{2\ell-1}) = -1+(\ell-1)/s. 
\]
{\it Case $i = 2\ell$ even}. If $i$ is even
\[
D^i = \det( \r, F_1, \dots, F_{i-1}, \wG_{2s+1}+\alpha_i^{2s+1}F_i, F_{i+1}, \dots, F_{2s})
\]
\[
=\det(\r, F_1, \dots, F_{i-1}, \wG_{2s+1}+\alpha_i^{2s+1}F_i, \wG_{i+1}+\alpha_i^{i+1}F_i, G_{i+2}, \wG_{i+3}+\alpha_i^{i+3}F_i, G_{i+4}, \dots, G_{2s}).
\]
This determinant breaks into one determinant with no $F_i$ and a sum of determinants with $F_i$ in different positions. The determinant with no $F_i$ vanishes since it is equal to
\[
\det(\r, \wG_1, G_2, \dots, G_{i-2}, \wG_{i-1}, \wG_{2s+1}, \wG_{i+1}, G_{i+2}, \dots, G_{2s})
\]
which contains more hat-vectors that the subspace they span. The remaining sums has one term where $F_i$ is in its original position (the one corresponding to $\wG_{2s+1}+\alpha_i^{2s+1}F_i$) and several terms with $F_i$ located in the remaining hat-entries. Thus, using (\ref{induction}) we have
\[
D^i = \alpha_i^{2s+1}D 
\]
\[+ \sum_{p=2q+1, q=0}^{s-\ell-1}\alpha_i^{p+i}\det(\r_0,\wG_1, \dots, \wG_{i-1}, \wG_{2s+1}, \wG_{i+1}, \dots, G_{i+p-1}, G_i, G_{i+p+1}, \dots, G_{2s})
\]
\[
= \alpha_i^{2s+1}D-\sum_{p=2q+1, q=0}^{s-\ell-1}\alpha_i^{p+i}\Delta_i^p
\]
where $\Delta_i^p = \det(\r, \wG_1, \dots, \wG_{i-1}, G_i, \wG_{i+1}, \dots, G_{i+p-1}, \wG_{2s+1}, G_{i+p+1}, \dots, G_{2s})$. 

Once again we can substitute $\wG_{2s+1}$ with $\wG_{2s+1}\otimes\wG_{i+p}^{-1}\otimes\wG_{i+p}$ and conclude that if $\wG_{2s+1}\otimes\wG_{i+p}^{-1}$ is homogeneous of degree $d_i^p$, then $\Delta_i^p$ will be homogeneous and $d(\Delta_i^p) = d(D)+d_i^p = d_i^p$. And indeed, from (\ref{gordereven}) we have that $d(g_{2s+1}\otimes g_{i+p}^{-1})$ is given by
\[
d(g_{2s+1})-d(g_{i+p}) = \begin{pmatrix}-1\\-1+1/s\\\vdots\\ -1/s\end{pmatrix} -   \begin{pmatrix}-(\ell+p)/s\\-(\ell+q-1)/s\\\vdots \\(s-1-\ell-q)/s\end{pmatrix} = \begin{pmatrix}-1+(\ell+q)/s\\-1+(\ell+q)/s\\\vdots\\ -1+(\ell+q)/s\end{pmatrix},
\]
and hence $d(\Delta_i^p) = -1 + (\ell+q)/s$. Our final step is to prove that $\alpha_i^j$ are homogeneous and to calculate their degrees. We know that $\alpha_0^{2r+1} = \left(\Sh G_{2r}\right)_{n+1} = \left[\Sh \wG_{2r-1}\right]_{n+1}^1$. Therefore, using (\ref{gordereven}) $d(\alpha_0^{2r+1}) = 1-r/s$. But, from equations (\ref{alpha1})-(\ref{alpha2}) we know that $d(\alpha_i^{p+i}) = d(\Sh^i\alpha_0^p) = d(\Sh^i\alpha_0^{2q+1}) = 1-q/s$. Therefore, \[
d(\alpha_i^{p+i} \Delta_i^p) = 1-q/s-1+(\ell+q)/s = \ell/s.
\]
We also have that 
\[
d(\alpha_i^{2s+1}) = d(\Sh^i\alpha_0^{2s-2\ell+1}) = 1-(s-\ell)/s = \ell/s
\]
which implies that $d(T(a^{2\ell})) = d(D^i) = \ell/s = d(a^{2\ell})$.
\end{proof}
\section{Discussion}
As we explained in the introduction and as explained in \cite{KS}, once the invariance under scaling is stablished, complete integrability (in the sense of existence of a Lax representation) follows. The authors of \cite{KS} showed that the generalization of the pentagram map studied in this paper is a discretization of the Boussinesq equation, or $(2,n+1)$-AGD flow - the same flow (realized in higher dimensions) that the original pentagram map is a realization of. It would be very interesting to investigate whether or not higher order AGD flows are also realized by integrable maps defined through the intersection of different subspaces. As shown in \cite{M1}, this is a non-trivial problem since in order to realize a higher dimensional flow one needs to break the very nice symmetry in the indices that the pentagram map has, and hence it is very unlikely that one would get a map that is invariant under a scaling. Still, invariance under scaling is only one possible technique to obtain a Lax representation, and other venues could be followed instead.


\begin{thebibliography}{9}
\bibitem{A} Adler  M., {\it On a Trace Functional for Formal Pseudo-differential Operators and the
Symplectic Structure of the KdV}, Invent. Math. {\bf 50}, 219-248 (1979).
 \bibitem{GD}  Gel'fand I.M. \& Dickey L.A.,  {\it A family of Hamiltonian structures connected with integrable
nonlinear differential equations}, in {\it I.M. Gelfand, Collected papers} v.1, Springer-Verlag, (1987).
\bibitem{KS} B. Khesin, F. Soloviev. {\em Integrability of higher pentagram maps}, arXiv:1204.0756.
  \bibitem{MMW} E. Mansfield, G. Mari-Beffa and J.P. Wang. {\em Discrete moving frames and discrete integrable systems}, to appear in Foundations of Computational Mathematics (FoCM).
  \bibitem{MW} G. Mari-Beffa, J.P. Wang. {\em Hamiltonian evolutions of twisted polygons in $\RP^n$}, submitted.
\bibitem{M1} G. Mar\'\i~Beffa, {\em On generalizations of the pentagram map: discretizations of AGD flows}, the Journal of Nonlinear Science, Dec (2012). 
\bibitem{M2} G. Mar\'\i~Beffa, {\em On bi-Hamiltonian flows and their realizations as curves in real semisimple homogeneous manifolds}, Pacific Journal of Mathematics, {\bf 247-1} (2010), pp 163-188.
\bibitem{OST} V. Ovsienko, R. Schwartz and S. Tabachnikov, {\em The Pentagram map: a discrete integrable system}, Communications in Mathematical Physics {\bf 299}, 409-446 (2010). 
{ \bibitem{OTS2} V. Ovsienko, R. Schwartz and S. Tabachnikov, {\em Liouville-Arnold integrability of the pentagram map on closed polygons}, arXiv:1107.3633.}
\bibitem{S1}R. Schwartz, {\em The Pentagram Map is Recurrent},
Journal of Experimental Mathematics, (2001) Vol {\bf 10.4} pp. 519-528.
\bibitem{S2} R. Schwartz, {\em The Pentagram Integrals for Poncelet Families} 
(2009) preprint.  
\bibitem{S3} R. Schwartz, {\em Discrete monodromy, pentagrams and the method of condensation}, J. of Fixed Point Theory and Appl. {\bf 3} (2008), 379-409. 
\bibitem{ST} R. Schwartz and S. Tabachnikov,  {\em The Pentagram Integrals for Inscribed Polygons} 
(2010) preprint. 
{ \bibitem{FS} F. Soloviev, {\em Integrability of the Pentagram Map}, arXiv:1106.3950.}
\bibitem{Wi} E.J. Wilczynski, {\em Projective differential geometry of curves and ruled surfaces}, 
 B.G. Teubner, Leipzig, 1906.



\end{thebibliography}
\end{document}